\documentclass[a4paper]{amsart}
\usepackage[english]{babel}
\usepackage[utf8]{inputenc}
\allowdisplaybreaks  

\usepackage{amssymb,xypic,amscd}

\usepackage{mathtools}
\usepackage{etoolbox}
\usepackage[colorlinks,linkcolor=Red3,citecolor=Red3,urlcolor=Cyan3]{hyperref}

\usepackage[svgnames,x11names,dvipsnames]{xcolor}

\newtheorem{thm}[equation]{Theorem}

\newtheorem{cor}[equation]{Corollary}

\newtheorem{prop}[equation]{Proposition}

\theoremstyle{remark}
\newtheorem{rem}[equation]{Remark}
\theoremstyle{remark}

\numberwithin{equation}{section} 

\newcommand\grv{{\operatorname{Gr}}(V)}

\newcommand\gr{\operatorname{Gr}}

\newcommand\Gl{\operatorname{Gl}}
\newcommand\gl{\operatorname{gl}}
\newcommand\g{{\mathfrak{g}}}
\newcommand\res{\operatorname{Res}}
\renewcommand\d{\operatorname{d}}
\renewcommand\det{\operatorname{det}}
\newcommand\Det{\operatorname{Det}}
\newcommand\Div{\operatorname{Div}}
\newcommand\tr{\operatorname{Tr}}

\renewcommand\div{\operatorname{div}}

\newcommand{\Lie}{\operatorname{Lie}}
\DeclarePairedDelimiterXPP\globalsymbol[1]{}{\langle}{\rangle}{}{\ifblank{#1}{\cdot,\cdot}{#1}}
\newcommand{\normSymbol}{N_{k(x)/k}}
\DeclarePairedDelimiterXPP\norm[1]{\normSymbol}{\lparen}{\rparen}{}{\ifblank{#1}{\cdot}{#1}}
\DeclarePairedDelimiterXPP\localsymbol[1]{}{\lparen}{\rparen}{}{\ifblank{#1}{\cdot,\cdot}{#1}}
\DeclarePairedDelimiterXPP\globalCommutator[1]{}{\lbrace}{\rbrace}{\raisebox{0pt}[1ex][1ex]{$^{\mathbb{A}_{X}}_{\mathbb{A}^{+}_{X}}$}}{\ifblank{#1}{\cdot,\cdot}{#1}}
\DeclarePairedDelimiterXPP\localCommutator[1]{}{\lbrace}{\rbrace}{^{K_{x}}_{\widehat{\O}_{X,x}}}{\ifblank{#1}{\cdot,\cdot}{#1}}

\renewcommand{\O}{{\mathcal{O}}}
\renewcommand{\k}{{\Bbbk}}
\renewcommand{\c}{{\bf{c}}}
\newcommand{\A}{{\mathbb{A}}}
\newcommand{\D}{{\mathbb{D}}}

\newcommand{\I}{{\mathbb{I}}}

\newcommand{\Z}{{\mathbb{Z}}}

\renewcommand\tilde{\widetilde}

\title[Weil Reciprocity Law and the Theorem of Residues] {Weil Reciprocity Law and the Theorem of Residues}

\author[J.~M.~Mu\~noz Porras]{José M. Mu\~noz Porras}

\author[F.~J.~PLaza~Mart\'in]{Francisco J. Plaza Mart\'in}
\email{fplaza@usal.es}

\thanks{Supported by grant MTM2015-66760-P of MINECO and  SA030G18 of JCyL.}

\address{Departamento de Matem\'aticas and IUFFYM, Universidad de
Salamanca,  Plaza de la Merced 1-4
        \\
        37008 Salamanca. Spain.
        \\
         Tel: +34 923294945. 
}


\begin{document}

\begin{abstract}
	This paper shows how the Theorem of Residues (TR) and the Gelfand-Fuchs cocycle can be deduced in a simple way from the Weil Reciprocity Law (WRL). Indeed, if one understand WRL as the triviality of certain extension of groups, then TR is the same statement at the level of Lie algebras. Finally, the Gelfand-Fuchs cocycle can also be obtained in this way. 
\end{abstract}

\maketitle

\section{Introduction}
\label{sec:introduction}

The algebraic theory of solitons (\cite{AMP}), which is a scheme theoretic approach to infinite grassmannians and loop groups (\cite{SegalWilson,Sato}), provides not only a natural framework to prove  the Weil Reciprocity Law (WRL) but also it allows us 
to replace groups by functors on groups. Accordingly,  by considering the associated Lie algebras (\cite{DemazureGabriel, SGA3}), we obtain simple and direct proofs of the Theorem of Residues (TR) and of the Gelfand-Fuchs cocycle. 

%
%

Our strategy for the WRL goes as follows. We fix a proper, irreducible, non singular algebraic curve $X$ over a base field $\k$ with function field $\Sigma_X$. Using infinite grassmannians, we construct a central extension of $\Sigma_X^*$ and we show that it is trivial (Theorem~\ref{t:WRL}). Hence, the commutator associated to it is trivial as well. However, this commutator can be explicitly written down in terms of the commutators of the central extensions of the fields $K_x:=(\widehat{\O}_{X,x})_{(0)}$ (Theorem~\ref{t:wx=1}). Putting everything together, the classical statement of the WRL follows (see~\eqref{e:CC}). It is worth pointing out that all these constructions are natural and that they do not requiere further structures; indeed, even the sign in WRL do show up canonically. 

Then, we focus on the TR, see  \S\ref{ss:TR}. Since  the formalism of \S\ref{sec:prelim} is valid for functors on groups and \S\ref{sec:expressions} contains a lot of explicit explicit expressions, we may repeat the same approach for the Lie algebra of $\Sigma_X^*$, $\Sigma_X=\Lie\Sigma_X^*$. As an instance of how the multiplicative setup (i.e. WRL for $\Sigma_X^*$) determines the additive setup (i.e. TR for $\Sigma_X$) see Corollary~\ref{c:cLie=tr}. 

Let us make some comments on future applications. First, our results are valid over any perfect field $\k$ and we hope to apply them in arithmetic. Second,  Corollary~\ref{c:Tate} shows how the definition of residues given by Tate \cite{Ta} can be deduced from WRL; indeed, we have been always very influenced by Tate's paper. Third, Theorem~\ref{t:algebraicity} shows that if an adele $\alpha\in\A_X$ verifies that $\sum_x \res_x \alpha\d g=0$ for all $g\in\Sigma_X$, then it is a rational function (up to an element of the radical of the residue pairing); thus, it is natural to wonder wether WRL can be used to characterize rational functions. Finally, \S\ref{ss:remarks} exhibits how the Gelfand-Fuchs cocycle shows up in this context and discuss briefly how groups, Lie algebras and their extensions appear in the geometric Langlands Program (\cite{Frenkel}). Although not very surprising, since geometric Langlands Program is an analog of class field theory (and, thus, WRL must be essential \cite{Serre,Weil}, it is a very exciting and promising connection that deserves further research. 

Few days before this paper was finished, José María, the first named  author, passed away. He was my advisor, collaborator and friend. As a brilliant mathematician and a very warm person, I will always remember these years together.  Let this paper be a small tribute to his memory.

\section{Preliminaries}\label{sec:prelim}

\subsection{Central Extensions}\label{subsec:extensions}

Let $\k$ be a d field. Let $(V,V^+)$ be a pair consisting of a $\k$-vector space and a subspace $V^+$. Recall that we say that two subspaces $A,B\subset V$ are commensurable, and we write $A\sim B$, when $(A+B)/(A\cap B)$ is a finite dimensional $\k$-vector space. If we regard the set of all subspaces $A\subset V$ commensurable with $V^+$, $A\sim V^+$, as a basis of neighborhoods of $0$, then $V$ carries a linear topology. We will assume that $V$ is separated, complete and that the topology induced in $V^+$ is discrete. Under these hypothesis  the Sato Universal Grassmannian Manifold (infinite Grassmaniann for short) does exist (see \cite{AMP,Plaza-Arit,Sato,SegalWilson}).

The following two instances will be fundamental along the paper:
\begin{itemize}
	\item \textsl{Local case} (ring of power series): $V=\k'(\!(z)\!), V^+=\k'[\![z]\!]$ with $\k\hookrightarrow \k'$ a finite extension. The topology is the $z$-adic topology. 
	\item \textsl{Global case} (adele ring): $V=\A_X$, $V^+=\A_X^+$ is the adele ring of a complete, irreducible and non singular algebraic curve over $\k$. 
\end{itemize}

For our purposes, it is convenient to recall from \cite{AMP,Plaza-Arit} that the infinite Grassmannian is a $\k$-scheme whose set of $\k$-rational points is
	\[
	\grv\,=\, \left\{ 
	\begin{gathered}
	U\subset V \text{ such that } U\to V/V^+
	\\
	 \text{has finite dimensional ker and coker}
	\end{gathered}\right\}\, .
	\]
Its connected components are labelled by the integers, $\grv =\coprod_{n\in \Z}\gr^n(V)$, and we say $U\in \gr^n(V)$ iff $n= \dim_{\k} (V^+ \cap U) - \dim_{\k} (V/V^+ + U)^*$.
The determinant line bundle $\D$ is defined on $\grv$ by
	\[ 
	\D\,:=\, \operatorname{Det}({\mathcal U}\to V/V^+)
	\] 
where ${\mathcal U}$ is the universal subspace and its fibre at a point $U$ is given by
	\[
	\wedge (V^+ \cap U) \otimes \wedge (V/V^+ + U)^*
	\] 
where $\wedge W$ denotes the exterior algebra of maximal degree of a finite dimensional $\k$-vector space $W$. This bundle will allow us to construct central extensions of groups as follows. Define the group
	\begin{equation}\label{eq:glv}
	\Gl(V) \,:=\, 
	 \left\{ 
	\begin{gathered}
	\text{$\k$-linear maps }  S:V\overset{\sim}{\to} V \text{ such that} 
	\\
	S(V^+)\sim V^+ \text{ and } S^{-1}(V^+)\sim V^+
	\end{gathered}\right\}
	\end{equation}
and note that it acts on $\grv$. Indeed, $S\in \Gl(V)$ acts on the grassmannian as $S^{-1}:\grv\to \grv$, which sends $U$ to $S^{-1}(U)$. 

Let us pursue how this action lifts to the determinant bundle. For any $A\sim V^+$, define $\D_A:= \operatorname{Det}({\mathcal U}\to V/A)$ and note that, by the theory of determinants of perfect complexes,  there is a canonical isomorphism
\[
\D_A \, \overset{\sim}{\to} \D \otimes \wedge ( V^+/(A\cap V^+)) \otimes \wedge( A/(A\cap V^+))^*
\]
In particular, for $S\in \Gl(V)$, one has 
	\[ 
	{(S^{-1})}^*\D\simeq \operatorname{Det}(S^{-1}{\mathcal{U}} \to V/V^+)
	\simeq \operatorname{Det}({\mathcal{U}} \to V/S(V^+)) = \D_{S(V^+)}
	\] 
and, therefore, there is a canonical isomorphism
	\begin{equation}\label{e:S*DD}
	{(S^{-1})}^*\D \, \overset{\sim}{\to} \D \otimes \wedge ( V^+/(S(V^+)\cap V^+)) \otimes \wedge( S(V^+) /(S(V^+)\cap V^+))^*
	\end{equation}
It follows that ${(S^{-1})}^*\D$ and $\D$ are isomorphic, although not canonically. 

Let us define  the \textit{winding number of $S$} to be
	\begin{equation}\label{eq:wind}
	w(S)\,:=\, \dim_{\k} (V^+/(V^+\cap S(V^+))) - \dim_{\k} (S(V^+)/(V^+\cap S(V^+)))\,. 
	\end{equation}
 and observe that $S^{-1}$ maps $\gr^n(V)$ to $\gr^{n+w(S)}(V)$ for all $n\in {\mathbb Z}$. Let $\Gl^0(V)$ be the normal subgroup of those maps with winding number $0$.

The fact that ${(S^{-1})}^*\D\simeq \D$ for $S\in\Gl(V)$  and that $H^0(\gr^0(V),\O^*)=\mathbb{G}_m$, where $\mathbb{G}_m$ denotes the multiplicative group, shows that the group
	\begin{equation}\label{eq:TildeGl0}
	\tilde{\Gl^0(V)}\,:=\,
	\left\{
		\text{pairs } (\sigma, S) \text{ where } S\in \Gl^0(V) \text{ and }
		\sigma: {(S^{-1})}^*\D \overset{\sim}{\longrightarrow}\D 
	\right\}
	\end{equation}
fits into the following exact sequence
	\begin{equation}\label{eq:centralext-global}
	0 \to \mathbb{G}_m  \to \tilde{\Gl^0(V)} \to \Gl^0(V) \to 0
	\end{equation}
where the pair $(\sigma, S)$ is mapped to $S$. Note that
	\[
\tilde{\Gl^0(V)}\,\simeq\,
\left\{
\text{diagrams } \raisebox{20pt}{\xymatrix@R=17pt{\D \ar[d] \ar[r]^{\sim} & \D \ar[d] \\ \gr^0(V) \ar[r]^{S^{-1}}_{\sim} & \gr^0(V)} }
\text{ s.t. } S\in \Gl^0(V)
\right\}
\]
and the composition law can be written down as
\begin{equation}\label{eq:GltildeLaw}
(\sigma,S)\cdot (\tau,T)\,:=\, (\sigma \circ {(S^{-1})}^*(\tau) , ST)
\end{equation}
since ${(S^{-1})}^*(\tau):{(S^{-1})}^*{(T^{-1})}^* \D \to {(S^{-1})}^*\D$ and $\sigma:{(S^{-1})}^* \D\to \D$. 

\begin{rem}\label{r:norm-det}
It is worth pointing out how the construction of \eqref{eq:centralext-global} behaves with respect to the base field. Let $\k'$ be a artinian $\k$-algebra. Let $(V,V^+)$ be a pair consisting of a $\k'$-module $V$ and a $\k'$-submodule of it $V^+$. On the one hand, regarding $\k'$ as base ring, the above procedure yields a central extension of $\Gl^0_{\k'}(V)$ by $\mathbb{G}_m (\k')=(\k')^*$ (see \cite{Plaza-Arit} for details). On the other hand, we may apply the previous construction to the pair of $\k$-vector spaces $(V,V')$ and get a central extension of $\Gl^0_{\k}(V)$ by $\mathbb{G}_m (\k)=\k^*$. It is not difficult to check that the pullback of the first extension by $\Gl^0_{\k'}(V) \hookrightarrow \Gl^0_{\k}(V)$ coincides with the pushout of the second extension by the norm $\operatorname{Norm}_{\k'/\k}:(\k')^*\to \k^*$. The key idea for the proof is the following result (\cite[III.9.4, Prop. 6]{Bourbaki}): let $T$ be an endomorphism of a trivial $\k'$-module of finite rank $V$ with determinant $\det_{\k'}T$; let $\det_{\k}T$ be the determinant of $T$ as endomorphism of $V$ as $\k$-vector space, then it holds that  $\operatorname{Norm}_{\k'/\k}(\det_{\k'}(T))=\det_{\k}(T)$. 
\end{rem}

In order to get a central extension of the whole group $\Gl(V)$ one has to deal with the following fact; the decomposition of the grassmannian into connected components implies that $\operatorname{Aut}_{\gr(V)}\D = H^0(\gr(V),\O^*)=\prod_{\mathbb{Z}} {\mathbb G}_m$. Thus, one may proceed as follows. Fix a pair $(\zeta, z)$ where $z\in\Gl^1(V)$ and $\zeta:{(z^{-1})}^*\D\overset{\sim}\to \D$ and consider 
	\begin{equation}\label{eq:TildeGl}
	\tilde{\Gl(V)}\,:=\,
	\left\{
		\begin{gathered}
		\text{pairs } (\sigma, S) \text{ where } S\in \Gl(V) \text{ and }
		\sigma: {(S^{-1})}^*\D \overset{\sim}{\longrightarrow}\D 
		\\
		\text{such that } (\sigma,S)\cdot (\zeta, z) \cdot (\sigma,S)^{-1} \cdot (\zeta, z)^{-1} \in {\mathbb{G}}_m
		\end{gathered}
	\right\}
	\end{equation}

%

\subsection{Group Law}\label{subsec:group-law}

Let us reinterpret the group law \eqref{eq:GltildeLaw}. 


For any pair $A\sim V^+$ and $M\sim V^+$ such that $M\subseteq A\cap V^+$, there are natural identifications  
	\begin{equation}\label{eq:wedge-VAM}
	\wedge ( V^+/(A\cap V^+)) \simeq  \wedge ( V^+/M)  \otimes \wedge ( (A\cap V^+)/M)^*
	\end{equation}
and 
	\begin{equation}\label{eq:wedge-AVM}
	\wedge ( A/(A\cap V^+))^* \simeq  \wedge ( (A\cap V^+)/M) \otimes \wedge ( A/M)^*  
	\end{equation}
so the pairing of  $\wedge ( (A\cap V^+)/M)$ and its dual induces a map
	\begin{equation}\label{eq:wedge-VAAV}
	\begin{aligned}
	\wedge ( V^+/(A\cap V^+)) \otimes \wedge ( A/(A\cap V^+))^* &\,\overset{\sim}\longrightarrow\, 
	\wedge ( V^+/M)  \otimes \wedge ( A/M)^* 
	\\
	v\otimes \alpha \otimes w\otimes\beta \quad & \,\longmapsto \, \alpha(w) v\otimes \beta
	\end{aligned}
	\end{equation}
where $v\otimes \alpha$ lies in \eqref{eq:wedge-VAM} and $w\otimes\beta$ in \eqref{eq:wedge-AVM}. Note that, if the elements have to be reordered before applying \eqref{eq:wedge-VAAV}, then the sign rule of the exterior algebra is in order. 

Accordingly, for any $M$ as above, there is also a canonical isomorphism 
\begin{equation}\label{eq:A-M}
\D_A \, \overset{\sim}{\to} \D \otimes \wedge ( V^+/M) \otimes \wedge( A/M)^*
\end{equation}

Let $\tilde{S}=(\sigma,S)\in \widetilde{G}$ be given. Having in mind \eqref{e:S*DD}, it follows that 
$\sigma: {(S^{-1})}^*\D \to \D$ can be interpreted as an element  
\begin{equation}\label{eq:wedge-sigma}
\sigma^- \otimes \sigma^+\in \wedge^{s^-} ( V^+/(S(V^+)\cap V^+))^* \otimes \wedge^{s^+}( S(V^+) /(S(V^+)\cap V^+))
.	\end{equation}

We compute the product of $\tilde{S}=(\sigma,S)$ and $\tilde{T}=(\tau,T)$ using the definition of the group law of $ \widetilde{G}$
\[
\tilde{S}\tilde{T} \, =\, (\sigma,S)(\tau,T) \,=\, 
( \sigma \circ {(S^{-1})}^*(\tau)   , ST)
\, .
\]
That is, $ \sigma \circ {(S^{-1})}^*(\tau)$ is the composite 
\[
{(S^{-1})}^* {(T^{-1})}^* \D \overset{{(S^{-1})}^*(t)}\longrightarrow    {(S^{-1})}^*\D \overset{\sigma}\longrightarrow  \D  
\, . 
\]

Arguing as above, ${(S^{-1})}^*(\tau)= {(S^{-1})}^*(\tau^-)\otimes {(S^{-1})}^*(\tau^+)  $ is thought as an element  of
	\[ 
 \wedge^{t^-} ( S(V^+)/(ST(V^+)\cap S(V^+)))^* \otimes \wedge^{t^+}( ST(V^+) /(ST(V^+)\cap S(V^+)))	
	\] 

Hence, $\sigma \circ {(S^{-1})}^*(\tau) $ is associated to the element of 
	\[
	\wedge ( V^+/(ST(V^+)\cap V^+))^* \otimes \wedge( ST(V^+) /(ST(V^+)\cap V^+))	
	\]
that is obtained from the tensor product of those elements associated to $\sigma$ and ${(S^{-1})}^*(\tau)$. In order to make such product rigorous, one has to consider $M\sim V^+$ such that $M\subseteq V^+ \cap S(V^+) \cap ST(V^+)$ and take into account the identification \eqref{eq:A-M}. Pursing these isomorphisms, one obtains that the element associated to  $\sigma \circ {(S^{-1})}^*(\tau)  $, which will be denoted by $\sigma \cdot  {(S^{-1})}^*(\tau)  $   is explicitly given by
	\begin{equation}\label{eq:product-sign}
	\sigma \cdot  {(S^{-1})}^*(\tau)
	\,:=\, 
	(-1)^{t^- s^+}\sigma^- \wedge {(S^{-1})}^*(\tau^-) \otimes \sigma^+ \wedge {(S^{-1})}^*(\tau^+)
	\, . \end{equation}

It is a straightforward calculation that this expression is well defined and that it does not depend on the choice of $M$. Indeed, let us choose $M$, $\mu^- \in \wedge^m (A\cap V^+)/M$  and $\mu^+ \in \wedge^m ((A\cap V^+)/M)^*$  such that $\mu^+(\mu^-)=1$. Then, we may represent $\sigma^- \otimes \sigma^+$ by $ \sigma^-  \wedge \mu^-\otimes \mu^+ \wedge \sigma^+$ and
	\[
	\begin{gathered}
	(\sigma^-  \wedge \mu^-\otimes \mu^+ \wedge \sigma^+) \,\cdot\, {(S^{-1})}^*(\tau^-)\otimes {(S^{-1})}^*(\tau^+) \,= \\
	=\, (-1)^{t^- (s^++m)} (\sigma^-  \wedge \mu^-   \wedge {(S^{-1})}^*(\tau^-) \otimes \mu^+ \wedge \sigma^+\wedge {(S^{-1})}^*(\tau^+) \,= \\
	=\, (-1)^{t^- (s^++m)} (-1)^{ m t^-} (\sigma^-   \wedge {(S^{-1})}^*(\tau^-)\wedge \mu^-   \otimes \mu^+ \wedge \sigma^+\wedge {(S^{-1})}^*(\tau^+) 
	\end{gathered}
	\]
which coincides with \eqref{eq:product-sign} since $\mu^+(\mu^-)=1$. Similarly, we can replace ${(S^{-1})}^*(\tau^-)\otimes {(S^{-1})}^*(\tau^+)$ by ${(S^{-1})}^*(\tau^-) \wedge \mu^-\otimes \mu^+ \wedge {(S^{-1})}^*(\tau^+) $ and check that the product is well defined.

\begin{rem}
It should be noted that the above discussion is an explicit version of the signs of the diagrams of \cite[\S4]{ACK}	and \cite[\S3.2]{AP}. On the other hand, this product resembles the tensor product of graded algebras. 
\end{rem} 
\subsection{Cocycles and Commutators}

From now on, we will fix a commutative subgroup $G$ of either $\Gl^0(V)$ or $\Gl(V)$. In both cases, using \eqref{eq:TildeGl0} or, respectively, \eqref{eq:TildeGl}, one obtains a central extension
	\begin{equation}\label{eq:centralext} 
	0 \to \mathbb{G}_m   \to \tilde{G} \to G \to 0
	\end{equation}

This central extension determines and it is determined by the class of a $2$-cocycle in $H^2(G, \mathbb{G}_m)$; that is, a map
	\begin{equation}\label{eq:cociclo}
	\c: G \times G \longrightarrow \mathbb{G}_m =\k^*
	\, .
	\end{equation}
Indeed, the cocycle provides an alternative description of the central extension and its group law. As a set, consider the product $\mathbb{G}_m\times G $ and endow it with the following composition law
	\[
	(s,S)\cdot (t,T)\,:=\, ( s t \cdot  \c(S,T), ST)
	\]
where we use the notation $(s,S)$ for an element of the semidirect product with this composition law and $(\sigma,S)$ for an element of $\tilde{G}$ with \eqref{eq:GltildeLaw} as composition law.

In particular, to each central extension of $G$ we associate its commutator
	\[
	\begin{aligned}
	G \times G & \longrightarrow {\mathbb G}_m \\
	(S_1, S_2) & \longmapsto  \langle S_1 , S_2 \rangle\,:=\, \tilde{S_1}\tilde{S_2}\tilde{S_1}^{-1}\tilde{S_2}^{-1} 
	\end{aligned}
	\]
where $\tilde{S_i}$ denotes a preimage of $S_i$ by the projection map $\tilde{G} \to G$. Note that the commutator, is skew-symmetric and multiplicative on both arguments (\cite[Lemma~3.1.2]{AP}). It then follows that the commutator is a skew-symmetric $2$-cocycle (\cite{Hosszu}).

Recall the well known relation between the commutator and the $2$-cocycle $\c$. 

\begin{prop}
Let $G$ be commutative and $S, T\in G$. Let $\c$ be any $2$-cocycle in the cohomology class associated to the central extension \eqref{eq:centralext}.  It holds that 	
	\begin{equation}\label{eq:Commutator}
	\langle S,T\rangle  \, = \, \frac{\c(S,T)}{\c(T,S)} 
	\end{equation}
\end{prop}

\begin{proof}
	It is straightforward from the properties of any $2$-cocycle. 
%
\end{proof}

\subsection{Extensions of Lie Algebras}\label{ss:Lie}

Bearing in mind the results of \cite{Plaza-Arit, Plaza-grass}, the preceding discussion can be generalized for the grassmannian and linear group of $(V\hat{\otimes}_{\k} R, V^+\hat{\otimes}_{\k} R)$ where $R$ is a local $\k$-algebra and  $\hat{\otimes}_{\k}$ denotes the completion w.r.t. the topology of $V$ given by $V^+$.  It follows that exact sequences \eqref{eq:centralext-global}, \eqref{eq:TildeGl} and  \eqref{eq:cociclo} can be still considered when objects are regarded as functors on groups defined on the category of local $\k$-algebras. 
	

It thus makes sense to consider the case $R=\k[\epsilon]/(\epsilon^2)$. Following \cite[Chap. II,\S4]{DemazureGabriel}, \cite[T I, Exp. 2]{SGA3} we may study the Lie algebras attached to a functor on groups $F$; indeed, it is defined by 
	\[
	\Lie F:=F(\k[\epsilon]/(\epsilon^2)\times_{F(\k)} \{1\}
	\] 
endowed with a natural Lie algebra structure
	\[
	\begin{aligned}
	\Lie F \times \Lie F & \longrightarrow \Lie F 
	\\
	(S, T) \, &\longmapsto [S,T]
	\end{aligned}
	\] 
arising from the adjoint action. More precisely, the Lie bracket of $S,T$ can be defined to be the element  $[S,T]\in \Lie F $ fulfilling the following defining relation 
	\[ 
	1 + \epsilon_1 \epsilon_2 [S,T] \,  =\, 
	(1+\epsilon_1 S)(1+ \epsilon_2 T) (1 + \epsilon_1 S)^{-1} (1+\epsilon_2 T)^{-1}
	\]	
where $\epsilon_1$ (resp. $\epsilon_2$) denotes $\epsilon\otimes 1$ (resp. $1\otimes\epsilon$) in $\k[\epsilon]/(\epsilon^2) \otimes_{\k} \k[\epsilon]/(\epsilon^2)$. This relation relies only on the group law of $F$. Note that, if one could expand (at least formally) the right hand side of the previous relation, one would obtain $1 + \epsilon_1\epsilon_2 (S T  - TS)$.


In our situation, recall that 
	\[
\Lie \mathbb{G}_m  \,:=\, \mathbb{G}_m(\k[\epsilon]/(\epsilon^2)) \times_{\mathbb{G}_m(\k)}\{\operatorname{Id}\} \,\simeq \, \k	
\] 
since one identifies $1+\epsilon a\in \Lie \mathbb{G}_m$ with $a\in \k$. Set $\g := \Lie G $. It is easy to conclude that \eqref{eq:centralext} yields an extension of Lie algebras
	\begin{equation}\label{eq:Ext-Lie}
		0 \to \k \to \tilde{\g} \to \g \to 0
	\end{equation}
whose Lie algebra structure is determined by the $2$-cocycle of Lie algebras as follows (\cite[Chap XIV]{CartanEilenberg}). Bearing in mind that the extension is central, the Lie bracket on the $\k$-vector space $\g\oplus \k$ is
	\[
	[ (s,S), (t,T) ] \,:=\, ( \c_{Lie}(S,T), [S,T] )
	\]
and, recalling  \eqref{eq:Commutator},  $\c$, $\c_{Lie}$ and the commutator $\langle\, , \,\rangle$ are related as follows
	\begin{equation}\label{eq:Lie-cocycle}
	1+ \epsilon_1 \epsilon_2 \c_{Lie}(S,T) 
		\,=\, \frac{\c(1+\epsilon_1 S , 1+ \epsilon_2 T)}{\c(1+\epsilon_2 T , 1 + \epsilon_1 S)}
		\,=\,   \langle  1+\epsilon_1 S, 1+ \epsilon_2 T  \rangle
	\end{equation}


\section{Explicit Expressions}\label{sec:expressions}

The $2$-cocycle $\c$ can be explicitly computed on a subset of $G$. Fix a decomposition $V=V^-\oplus V^+$ and observe that, in particular, $V^-\in \gr^0(V)$. Accordingly, any element of $\Gl(V)$ admits a block decomposition;  write  $S_i=\begin{pmatrix}	\alpha_i & \beta_i \\ \gamma_i & \delta_i 	\end{pmatrix}$.

Let us recall that for a trace class operator $A:V^+\to V^+$, one defines its determinant by $\det (\operatorname{Id}+A)=\sum_{i=0}^{\infty} \operatorname{Tr}\wedge^i A$. We address de reader to  \cite{Gro,Simon} for details. Further, for a trace class operator $A$ it holds that $\operatorname{Id}+A$ is invertible if and only if its determinant does not vanish. This notion of determinant can be extended for those operators $A$ such that $A^n$ is of finite rank for some $n$ (see \cite{HP}). For instance, if $V^+=\k[[z]]$ and  $A:V^+\to V^+$ is the homothety of ratio $1+ a z^n$ with $n>0$, then $\det (\operatorname{Id}+A)$ is well defined and it is equal to $1$.  Similarly, the composite of the homothety $(1+ a z^n): V^+\to V$, with $n<0$ with the projection $V\to V/z^{-1}\k[z^{-1}]\simeq V^+$ has also a well defined determinant. The following two results are algebraic counterparts of those of \cite[\S3]{SegalWilson}. 

%

\begin{prop}\label{p:cocycle}
Let 
	\[
	G^*\, :=\, 
	\Big\{ S= \begin{pmatrix}	\alpha & \beta \\ \gamma & \delta 	\end{pmatrix}\text{ s.t. $\delta$ is invertible} \Big\}
	\,\subseteq \, G\, . 
	\]
Up to a coboundary, it holds that	
	\[
	\c( S_1 , S_2 )\,=\,\det\big(\delta_1 \delta_2 (\gamma_1\beta_2+\delta_1\delta_2)^{-1}\big)
	\]
	for all $S_1, S_2 \in G^*$ such that $S_3:=S_1\circ S_2\in G^*$.
\end{prop}

\begin{proof}
The cocycle can be computed through the following procedure. We fix a set-theoretic section $\rho:  G\to \widetilde{G}$ of the projection of \eqref{eq:centralext}. Then, the cocycle fulfills the following relation
	\[
	\rho(S_1) \rho(S_2) \,=\, \c(S_1,S_2) \rho(S_3) 
	\]
or $\c(S_1,S_2) =\rho(S_1) \rho(S_2)  \rho(S_3)^{-1}$. Hence, the proof consists of finding a section $\rho$ defined on $G^*$. That is, to each $S\in G^*$ we associate an element $\rho(S)=(\sigma,S) \in \widetilde{G}$. 

Let $S\in \Gl(V)$ and consider its block decomposition as above. Then, there are natural isomorphisms
	\[
	V^+ / \ker(\beta) \,\overset{\sim}\to\, \operatorname{Im}(\beta) \,\overset{\sim}\to\, (V^++S(V^+))/V^+ \,.
	\]
Since $S\in \Gl(V)$ (see \eqref{eq:glv}), these three vector spaces are finite dimensional. That is, $V^+\sim  \ker(\beta)$ or, what amounts to the same, $\beta$ is of finite rank. 

Let $S\in G^*$. Having in mind \S\ref{subsec:group-law} and the properties of the determinant of a perfect complex, there are canonical isomorphisms
	\[
	{(S^{-1})}^*\D \,\simeq\, \Det\big(S^{-1}{\mathcal{U}}\oplus\ker(\beta)\to V\big)  \otimes \wedge V^+/\ker(\beta)
	\]
and
	\[
	\D\,\simeq \, \Det\big({\mathcal{U}}\oplus\delta(\ker(\beta))\to V\big)  \otimes \wedge V^+/\delta(\ker(\beta))\, . 
	\]
Observe now that the fact that $\delta$ is an isomorphism implies that: a) $S$ yields an isomorphism from the complex $S^{-1}{\mathcal{U}}\oplus\ker(\beta)\to V$ to the complex ${\mathcal{U}}\oplus\delta(\ker(\beta))\to V$; and, b)   $\delta$ gives an  isomorphism $\bar{\delta}: V^+/\ker(\beta) \overset{\sim}\to V^+/\delta(\ker(\beta))$ (note that these vector spaces are of finite dimension). Accordingly, we obtain an isomorphism ${(S^{-1})}^* \D \overset{\sim}\to  \D$ which will be denoted by $\det(\bar{\delta})$. The desired section is defined by sending $S$ to $\rho(S)=(\det(\bar{\delta}),S)$.

Finally, for $S_i= \begin{pmatrix}	\alpha_i & \beta_i \\ \gamma_i & \delta_i 	\end{pmatrix} \in G^*$ as in the statement, one has that $\delta_3= (\gamma_1 \beta_2+ \delta_1 \delta_2)$. Having in mind that the composition
\begin{equation}\label{e:composition}
V^+ \, \overset{\delta_1^{-1}}{\longrightarrow} \, V^+ 
\, \overset{\delta_2^{-1}}{\longrightarrow} \, V^+ 
\, \overset{(\gamma_1 \beta_2+ \delta_1 \delta_2)}{\longrightarrow} \, V^+ \end{equation}
has a well defined determinant (since it is the identity on $\delta_1(\delta_2(\ker(\beta_2)))$ and $	V^+ / \delta_1(\delta_2(\ker(\beta_2)))$ is finite dimensional), one checks straightforwardly that $\c(S_1,S_2) = \det\big(\delta_1 \delta_2 (\gamma_1\beta_2+\delta_1\delta_2)^{-1}\big)$.
\end{proof}
	
\begin{cor}\label{c:cLie=tr}
	Let  $S_i=\begin{pmatrix}	\alpha_i & \beta_i \\ \gamma_i & \delta_i 	\end{pmatrix} \in \g$ for $i=1,2$. It  then holds that 	
		\[
		\c_{Lie}(S_1,S_2) \,=\,   \tr(\gamma_2 \beta_1- \gamma_1\beta_2)\,.
		\]
\end{cor}

\begin{proof}
Note that $1+\epsilon_i S_i = \begin{pmatrix}	1+\epsilon_i \alpha_i & \beta_i \\ \gamma_i & 1+\epsilon_i  \delta_i 	\end{pmatrix}\in\g $ is a point of $G^*$ with values in $\k[\epsilon_i]/\epsilon_i^2$ and, applying Proposition~\ref{p:cocycle}, we obtain
	\[
	\begin{aligned}
	\c(1+ & \epsilon_1 S_1   , 1+\epsilon_2 S_2) \, =\\ &= \,
	\det\big((1+\epsilon_1  \delta_1 ) (1+\epsilon_2  \delta_2) ( \epsilon_1\epsilon_2\gamma_1\beta_2 +(1+\epsilon_1  \delta_1 ) (1+\epsilon_2  \delta_2))^{-1}\big)
	\, =\\ &= \,
	\det\big((1+\epsilon_1  \delta_1 +\epsilon_2  \delta_2 +\epsilon_1  \epsilon_2 \delta_1  \delta_2) (1- \epsilon_1  \delta_1 -\epsilon_2  \delta_2  +  \epsilon_1\epsilon_2(  \delta_2\delta_1 - \gamma_1\beta_2 )\big)
		\, =\\ &= \,
	1 - \epsilon_1\epsilon_2 \tr(\gamma_1\beta_2 )\, . 
	\end{aligned}
	\]
Similarly, one gets that $\c(1+  \epsilon_2 S_2   , 1+\epsilon_1 S_1) = 1 - \epsilon_1\epsilon_2 \tr(\gamma_2\beta_1 )$. Recalling~\eqref{eq:Lie-cocycle}, and that $\tr$ is additive on the space of endomorphisms of finite rank (\cite{Ta}), one concludes.

\end{proof}
%

\subsection{Ring of Power Series}

Let $\k'$ be an artinian $\k$-algebra and $V=\k'((z))$, $V^+=\k'[[z]]$. Let $G=\k'((z))^*$. Let $\nu:V\setminus\{0\}\to \Z$ be the $z$-adic valuation. For a series $g\in G$, the winding number of the homothety defined by $g$ fulfills
	\[
	w(g)\,=\, \nu(g)\cdot  \dim_{\k}\k'
	\, . \]

\begin{prop}\label{p:Com-cocycle}
	Let $S,T\in G$.  It holds that 	
	\begin{equation}\label{eq:Comm-cocycle}
	\langle S,T\rangle  \, = \, \operatorname{Norm}_{\k'/\k} \Big( \frac{S^{\nu(T)}}{T^{\nu(S)}}(0)\Big)
	\end{equation}
where, for $S\in \k'[[z]]$, $S(0)$ denotes the class of $S$ in $\k'[[z]]/z\k'[[z]]$.
\end{prop}

\begin{proof}
For the sake of clarity, let us assume first that $\k=\k'$. Observe that the commutator takes values in the center of $\tilde{G}$
		\[
		\langle S,T\rangle  \, := \, 
		\tilde{S}\tilde{T}\tilde{S}^{-1}\tilde{T}^{-1} \, \in \,   \operatorname{Z}(\tilde{G}) = \mathbb{G}_m \, . 
		\]

Let us write $\tilde{S}=(\sigma,S)$, $\tilde{T}=(\tau,T) \in \tilde{G}$. The commutator fulfills the relation
	\[
	\tilde{S}\tilde{T} \, =\, \langle S,T\rangle   \tilde{T} \tilde{S} \, . 
	\]
We compute the products using the definition of the group law
	\[
	\tilde{S}\tilde{T} \, =\, (\sigma,S)(\tau,T) \,=\, 
	( \sigma \circ {(S^{-1})}^*(\tau)   , ST)
	\]
	\[
	\tilde{T}\tilde{S}  \, =\, (\tau,T) (\sigma,S) \,=\, 
	(\tau \circ {(T^{-1})}^*(\sigma)  , T S)
	\]
and therefore $ \sigma \circ {(S^{-1})}^*(\tau) = \langle S,T\rangle  \tau \circ {(T^{-1})}^*(\sigma)  $; that is, since $S$ and $T$ commute, the following two morphisms 
	\[
	\begin{gathered} 
	{(S^{-1})}^* {(T^{-1})}^* \D \overset{{(S^{-1})}^*(t)}\longrightarrow    {(S^{-1})}^*\D \overset{\sigma}\longrightarrow  \D  
	\\
	{(T^{-1})}^* {(S^{-1})}^* \D \overset{{(T^{-1})}^*(\sigma)}\longrightarrow {(T^{-1})}^*\D \overset{t}\longrightarrow \D
	\end{gathered}
	\]
coincide up to multiplication by  $\langle S,T\rangle$. In order to compare these morphisms, we will compute the elements associated to them in the corresponding exterior algebras as it was discussed in \S\ref{subsec:extensions}
. 

Recall that  $\sigma: {(S^{-1})}^*\D \to \D$ can be interpreted as an element of 
	\begin{equation}\label{eq:wedge3}
	\wedge ( V^+/(S(V^+)\cap V^+))^* \otimes \wedge( S(V^+) /(S(V^+)\cap V^+))
	\end{equation}
and $\tau: {(T^{-1})}^*\D \to \D$ is an element of 
	\begin{equation}\label{eq:wedge1}
	\wedge ( V^+/(T(V^+)\cap V^+))^* \otimes \wedge( T(V^+) /(T(V^+)\cap V^+))
	\end{equation}
Accordingly, ${(T^{-1})}^*(\sigma)$ lies in
	\begin{equation}\label{eq:wedge2}
	 \wedge ( T(V^+)/(TS(V^+)\cap T(V^+)))^* \otimes \wedge( TS(V^+) /(TS(V^+)\cap T(V^+)))
	\end{equation}
and ${(S^{-1})}^*(\tau)$ belongs to
	\begin{equation}\label{eq:wedge4}
	 \wedge ( S(V^+)/(ST(V^+)\cap S(V^+)))^* \otimes \wedge( ST(V^+) /(ST(V^+)\cap S(V^+)))	 \, . 
	\end{equation}
Hence, $\tau \cdot {(T^{-1})}^*(\sigma) $ is an element of the tensor product of \eqref{eq:wedge1} and \eqref{eq:wedge2} while $\sigma \cdot {(S^{-1})}^*(\tau)$ belongs to the tensor product \eqref{eq:wedge3}$\otimes$\eqref{eq:wedge4}. 


Using the properties of the commutator, we reduce the proof to a bunch of cases as follows. Indeed, for $S,T\in G=\k((z))^*$, one can find unique expressions 
	\[
	\begin{gathered}
	S\,=\, s_0 z^s \prod_{i>0}(1 +s_i z^i) \\
	 T\,=\, t_0 z^t \prod_{i>0}(1 +t_i z^i)
	 \end{gathered}
	 \]
where $s_0,t_0\in \k^*$, $s,t\in {\mathbb Z}$, $s_i,t_i\in \k$. 

Bearing in mind that the commutator  is skew-symmetric and bi-multiplicative, it suffices to check the following cases.

\textit{Case 1}. $\langle s_0, z^t\rangle$ with $t>0$. In this case we must find a canonical isomorphism from $\wedge (V^+ / z^t V^+)^*$ to  $\wedge(s_0 V^+ / s_0 z^t V^+)^*$. It is clear that the identity does the job and that, choosing a basis in $V^+$ and those induced in these quotients, the determinant of the identity is the homothety $s_0^t$; that is, we have shown $\langle s_0, z^t\rangle= s_0^t$. 

\textit{Case 2}. $\langle 1 +s_i z^i, z^t\rangle$ with $t>0$. Observe that $ 1 +s_i z^i$ is invertible and $( 1 +s_i z^i)V^+=V^+$ since $i>0$. Again, the identity of $V^+$ induces an isomorphism from $\wedge (V^+ / z^t V^+)^*$ to  $\wedge(( 1 +s_i) V^+ / ( 1 +s_i) z^t V^+)^*$. The latter isomorphism is the determinant, which is $1$; that is, $\langle 1 +s_i z^i, z^t\rangle=1$.

\textit{Case 3}. $\langle z^s, z^t\rangle$ with $t\geq s>0$. The isomorphism from 
$\wedge (z^t V^+/z^{s+t}V^+)^* \otimes \wedge (V^+/z^tV^+)^*$ to $\wedge (z^s V^+/ z^{s+t} V^+)^* \otimes \wedge(V^+/z^s V^+)^*  $  induced by the identity, which is a reordering, is $(-1)^{(t-s)(s+s)}=1$. 

\textit{Case 4}. $\langle 1 +s_i z^i, t_0\rangle$ and $\langle 1 +s_i z^i, 1+t_j z^j\rangle$. Bearing in mind \eqref{e:S*DD} and that $(1+s_i z^i)V^+=V^+$, it follows that there exists a canonical isomorphism $((1+s_i z^i)^{-1})^*\D\simeq \D$. Analogously, $(t_0^{-1})^*\D\simeq \D$ canonically. Thus $\langle 1 +s_i z^i, t_0\rangle=1$. The same arguments show that $\langle 1 +s_i z^i, 1+t_j z^j\rangle=1$.

For the general case, $\k'$ arbitrary,  the same arguments show that $\langle s_0, z^t\rangle= \operatorname{Norm}_{\k'/\k}(s_0)^{v_z(T)}$. See also Remark~\ref{r:norm-det}.  

Putting everything together, the claim follows.
\end{proof}


Let $(V=\k((z)),V^+=\k[[z]])$ and $G=\k((z))^*$. Let $R$ be a local artinian ring with maximal ideal $\mathfrak{m}$ and residue field $\k$ and recall that the set of $R$-valued points of $G$ is $G(R)=R((z))^*$. Then, consider the subfunctor of $G^*$ (recall the definition of $G^*$ from Proposition~\ref{p:cocycle}) defined by
\[
G_1(R)\,:=\, G(R) \times_{G(\k)}\{1\} \,\subseteq \, G^*(R)  \, . 
\]
Recall that elements of $G_1(R)$ can be expressed as products
    \[G_1(R)\,=\, 
    \left\{ 
    \begin{gathered}
    \prod_{i>0} (1-\bar a_i z^{-i}) \prod_{i\geq 0} (1-a_i z^i)\in R((z)) \text{ such that } a_i\in\mathfrak{m} \,\forall i> 0, 
    \\
    \text{$\bar a_i=0$ for almost all $i<0$ and $\bar a_i$ nilpotent for all $i<0$}
    \end{gathered} \right\}\]
Bearing in mind the construction of the central extension as well as Remark~\ref{r:norm-det}, we obtain a central extension of the group $G_1(R)$ by $\k^*$. 

\begin{prop}\label{p:Con-Carr}
	Let $f= \prod (1-\bar a_i z^{-i}) \prod (1-a_i z^i)\in R((z)) ,g=\prod (1-\bar b_i z^{-i}) \prod (1-b_i z^i)  \in G_1(R)$. Then, the commutator is
	\[
	\langle f, g \rangle\,=\, 
	\operatorname{Norm}_{R/\k} \left( \frac{\prod_{i,j}\big(1-  a_i^{\frac{j}{(i,j)}}\bar b_j^{\frac{i}{(i,j)}} \big)^{(i,j)}}
	{\prod_{i,j}\big(1-  b_i^{\frac{j}{(i,j)}}\bar a_j^{\frac{i}{(i,j)}} \big)^{(i,j)}}
	\right)
	\]
where the products run over $i,j>0$. 
\end{prop}

\begin{proof}
	Since the commutator is skew-symmetric and bi-multiplicative, it suffices to compute the following cases.

\textit{Case 1}. $f=1-a z^{m} $ and $g= 1 - \bar b z^{-n}$ where $m>0, n>0$, $a\in \mathfrak{m}$ and $\bar b$ is nilpotent; that is, one has to compute the expression
		\[
		\langle 1-a z^{m} , 1 - \bar b z^{-n} \rangle \, . 
		\]
Bearing in mind Proposition~\ref{p:cocycle} we obtain that 
	\[
	\langle f, g \rangle\,=\,  \frac{\c(f,g)}{\c(g,f)} \,=\, 
	\operatorname{Norm}_{R/\k} \Big( \frac{\det_R(\delta_1 \delta_2(\gamma_1\beta_2+ \delta_1 \delta_2)^{-1})}
	{\det_R (\delta_2 \delta_1 (\gamma_2\beta_1+ \delta_2 \delta_1)^{-1})}\Big) 
	\]
where $S_1$ (resp. $S_2$) is the homotethy of ratio $f$ (resp. $g$). Observe that $\gamma_2=0$ and, thus, $\delta_3 = \gamma_1\beta_2+ \delta_1 \delta_2 = \delta_2 \delta_1$. Then, the denominator equals $1$ and the commutator reduces to the inverse of the determinant of \eqref{e:composition} or, what is tantamount, $\det_R(\delta_1 \delta_2  \delta_1^{-1} \delta_2^{-1})$.   A careful but  straightforward computation shows that the matrix associated to $\delta_1 \delta_2  \delta_1^{-1} \delta_2^{-1}\vert_{V^+}$ with respecto to the basis $\{z^i\vert i\geq 0\}$ is given by $M=(m_{ij})$ where
	\[
	m_{ij}\,=\, \delta_{i,j} -
	\begin{cases}
	 a\bar b^{\lceil \frac{j}{n}\rceil} & \text{ if } m-n< i \leq m
	 \\
	 0 & \text{ otherwise}
	\end{cases}
	\]
where $\delta_{i,j} $ is $1$ for $i=j$ and $0$ otherwise; and $\lceil x \rceil$ denotes the lowest integer equal or bigger than (the rational number) $x$. 

Further, $M$ acquires the following expression:
	\[
	M\,=\, \left(\begin{array}{c|c|c}
	\operatorname{Id}_{m-n} & 0 & 0
	\\ \hline
	\ast & N & \ast 
	\\ \hline
	0 & 0 & \begin{matrix} 1 & 0 \\ 0 & \ddots \end{matrix}
	\end{array}\right)
	\]
where the top raw does not appear for $m-n\leq 0$. Thus, $\det_R(M)$ exists and it is equal to  the determinant of the $n\times n$-matrix $N$. Since $\det_R(N)=(1- a^{\frac{n}{(m,n)}} \bar b^{ \frac{m}{(m,n)}})^{(m,n)}$, we are done.

\textit{Case 2}. $f=1-a $ and $g= 1 - \bar b z^{-n}$ where $n>0$, $a\in \mathfrak{m}$ and $\bar b$ is nilpotent or $f=1-a z^{m} $ and $g= 1 - \bar b$ where $m>0$, $a\in \mathfrak{m}$ and $\bar b$ is nilpotent. In both situations, proceeding as in the previous case, one obtains that $M$ is the identity matrix, whose determinant is $1$. 

\textit{Case 3}.  
$f=1-\bar a z^{-m} $ and $g= 1 - \bar b z^{-n}$ where $m>0, n>0$, $\bar a\,\bar b$ are nilpotent. Applying Proposition~\ref{p:cocycle} as in the previous case, and noting that $\gamma_1=\gamma_2=0$, it follows that $\langle 1-\bar a z^{-m} , 1 - \bar b z^{-n} \rangle=1$.

\textit{Case 4}.  
$f=1-a z^{m} $ and $g= 1 - b z^{n}$ where $m>0, n>0$, $a, b\in \mathfrak{m}$. Applying Proposition~\ref{p:cocycle}  and noting that $\beta_1=\beta_2=0$, it follows that $\langle 1-a z^{m} , 1 - b z^{n} \rangle=1$.

\end{proof}



\subsection{Ring of Adeles}\label{ss:adeles}

Let $X$ be a proper, irreducible, non singular algebraic curve over a perfect field $\k$. Let $\Sigma_X$ be its function field. Let $\k$ be algebraically closed in $\Sigma_X$. 

For each closed point $x\in X$, let $A_x:=\widehat{\O}_{X,x}$ and $K_x:=(A_x)_{(0)}$. Let $\A_X := \prod'_{x\in X} K_x$ be the adele ring of $X$ and let $\A_X^+ := \prod_{x\in X} A_x$. We can consider the grassmannian of the pair $(\A_X,\A_X^+)$ as well as the constructions of \S\ref{subsec:extensions}
.

Let $\I_X$ denote the idele group of $\A_X$; that is, the group of invertible elements. Note that given $\alpha\in \I_X$, it make sense to consider the associated divisor $\div(\alpha):=\sum_{x\in X} v_x(\alpha)x \in \Div(X)$ since it is a finite sum. Observe that $\I_X = \Gl(1,\A_X)$ and that an idele acts on $\A_X$ by multiplication. For each $\alpha$, the degree of $\div(\alpha)$ coincides with the winding number of the homothety defined by it. Set $\I_X^0= \Gl^0(1,\A_X)$ be the subgroup of $\I_X$ of those ideles of degree $0$.  

Let $G$ be a subgroup of $\I_X^0$. 

\begin{thm}
Suppose we are given $S_x\in\Gl(K_x)$ such that $S:=\prod_x S_x\in \Gl^0(\A_X)$. Then $w_x(S_s)$, the winding number of $S_x\in\Gl(K_x,A_x)$, is $0$ for almost all $x\in X$ and 
	\[
	 \sum_{x\in X} w_x(S_x) \,=\, 0 \, . 
	\]
\end{thm}

\begin{proof}
The fact that $S=\prod_x S_x$ implies that $S(V^+)=\prod_x S_x(A_x)$ and, thus, it holds that
	\[
	 V^+/(V^+\cap S(V^+)) \,=\, \prod_{x\in X} (A_x/(A_x\cap S_x(A_x))) \, . 
	\]
Since the l.h.s. is a finite dimensional vector space, it follows that each term of the r.h.s. vanishes for almost all $x\in X$. Arguing similarly with $S(V^+)/(V^+\cap S(V^+))$, and taking their dimensions as $\k$-vector spaces, the result follows.  
\end{proof}

\begin{thm}\label{t:wx=1}
If $S=\prod_x S_x, T = \prod_x T_x \in \Gl^0(\A_X)$ are as above, then $ \langle S_x, T_x\rangle =1$ for almost all $x\in X$ and
	\begin{equation}\label{eq:dirsum-cocycle}
	\langle S,T \rangle \,=\, \prod_{x\in X} (-1)^{w_x(S) w_x(T)} \langle S_x, T_x \rangle \, . 
	\end{equation}
\end{thm}

\begin{proof}
Similarly to the proof of Proposition~\ref{p:Com-cocycle}, for computing the commutator we compare $\tilde S\tilde T$ and $\tilde T\tilde S$. Let $\tilde S=(\sigma, S)=\prod (\sigma_x, S_x)$ and $\tilde T=(\tau, T)=\prod (\tau_x, T_x)$.  Let $\sigma_x^+\otimes \sigma_x^-$ be the element of \eqref{eq:wedge3} associated to $\sigma$ consisting of a $s_x^+$-form times a $s_x^-$-vector where $s_x^+:=\dim_{\k} A_x/(A_x\cap S_x(A_x)))$ and $s_x^- := \dim_{\k} S_x(A_x)/(A_x\cap S_x(A_x)))$. Let $\tau_x^+\otimes \tau_x^-$ be the element associated to $\tau$ and define $t_x^+$, $t_x^-$ analogously. 

Let us consider
	\[
	X^0\,:=\, \{ x\in X \text{ s.t. } s_x^+=s_x^-=t_x^+=t_x^-=0\}
	\]
and note that if $x\in X^0$, then $w_x(S_x)=w_x(T_x)=0$ and $\langle S_x, T_x \rangle=1$ by Proposition~\ref{p:Com-cocycle}.  Since Theorem~\ref{t:wx=1} states that $X\setminus X^0$ is finite, the first claim follows. 

Set $\{x_1,\dots, x_n\}=X\setminus X^0$. First, we will compare 
	\[
	\left( \otimes_{i=1}^n  \left( \sigma_{x_i}^+\otimes \sigma_{x_i}^-\right)\right)   \cdot  \left( \otimes_{i=1}^n  {(S^{-1}_{x_i})}^*(\tau_{x_i}^+\otimes \tau_{x_i}^-)\right)
	\]
with (recall the expression on the product \eqref{eq:product-sign})
	\[
	 \otimes_{i=1}^n  \left( \sigma_{x_i}^+\otimes \sigma_{x_i}^- \,\cdot\,  {(S^{-1}_{x_i})}^*(\tau_{x_i}^+\otimes \tau_{x_i}^-)\right)
	\, . \]

The sign arising from the definition of $\cdot$ is equal to $-1$ to the power
	\[
	(\sum_{i=1}^n t^+_{x_i} )(\sum_{i=1}^n s^-_{x_i} ) - (\sum_{i=1}^n t^+_{x_i} s^-_{x_i})
	\]
Since $\sigma_{x_i}^+\wedge {(S^{-1}_{x_j})}^*(\tau_{x_j}^+) = (-1)^{s^+_{x_i} t^+_{x_j}} {(S^{-1}_{x_j})}^*(\tau_{x_j}^+) \wedge \sigma_{x_i}^+ $, the sign arising from reordering is equal to $-1$ to the power	
	\[
	\sum_{i=1}^n ( s^+_{x_i}  \cdot \sum_{j=i+1}^n t^+_{x_j} ) + \sum_{i=1}^n ( s^-_{x_i} \cdot \sum_{j=i+1}^n t^-_{x_j} )
	\]

Second, similar arguments show that $\left( \otimes_{i=1}^n  \left( \tau_{x_i}^+\otimes \tau_{x_i}^-\right)\right)   \cdot  \left( \otimes_{i=1}^n  {(T^{-1}_{x_i})}^*(\sigma_{x_i}^+\otimes \sigma_{x_i}^-)\right)$ differs from $\otimes_{i=1}^n  \left( \tau_{x_i}^+\otimes \tau_{x_i}^- \,\cdot\,  {(T^{-1}_{x_i})}^*(\sigma_{x_i}^+\otimes \sigma_{x_i}^-)\right)$ is $(-1)$ to the power
	\[
	(\sum_{i=1}^n s^+_{x_i} )(\sum_{i=1}^n t^-_{x_i} ) - (\sum_{i=1}^n s^+_{x_i} t^-_{x_i}) +
	\sum_{i=1}^n ( t^+_{x_i}  \cdot \sum_{j=i+1}^n s^+_{x_j} ) + \sum_{i=1}^n ( t^-_{x_i} \cdot \sum_{j=i+1}^n s^-_{x_j} ) 
	\]
	
Sampling everything together, the sign is $(-1)$ to the power
	\[
	\sum_{i=1}^n  s^+_{x_i} t^+_{x_i}  +  s^+_{x_i} t^-_{x_i}  +  s^-_{x_i} t^+_{x_i}  +  s^-_{x_i} t^-_{x_i} 
	\,=\, \sum_{i=1}^n w_{x_i}(S) w_{x_i}(T)
	\]	
	
Finally, the difference between  $\otimes_{i=1}^n  \left( \sigma_{x_i}^+\otimes \sigma_{x_i}^- \,\cdot\,  {(S^{-1}_{x_i})}^*(\tau_{x_i}^+\otimes \tau_{x_i}^-)\right)$  and $ \otimes_{i=1}^n  \left( \tau_{x_i}^+\otimes \tau_{x_i}^- \,\cdot\,  {(T^{-1}_{x_i})}^*(\sigma_{x_i}^+\otimes \sigma_{x_i}^-)\right)$ is precisely the commutator 	
	\[
	\langle  \sigma_{x_i}^+\otimes \sigma_{x_i}^-  \, , \, \tau_{x_i}^+\otimes \tau_{x_i}^- \rangle
	\,=\, 
	\langle S_{x_i} \, , \, T_{x_i}\rangle
	\]
and the result follows. 
\end{proof}

\section{Pairing,Weil and Residues}

We continue with the notation introduced in \S\ref{ss:adeles}. That is,  $X$ is a proper, irreducible, non singular algebraic curve over $\k$, $\A_X$ is the adele ring, etc. 

We address the reader to \cite{Weil} for a connection of WRL and class field theory and to \cite{Serre} for an approach to local symbols for algebraic curves. 

%
%

\subsection{Weil Reciprocity Law}\label{ss:WRL}

We consider the central extension 
\begin{equation}\label{eq:Idele-ext}
0\to {\mathbb{G}_m} \to \tilde{\I}^0_X \to \I^0_X \to 0
\end{equation}
constructed from \eqref{eq:centralext-global} through $\I^0_X\hookrightarrow \Gl^0(1,\A_X)$. Further, since  $\I_X^0$ is commutative, it makes sense to consider the pairing defined by the commutator; that is, 
\begin{equation}\label{e:pairing-I}
\begin{aligned}
\I_X^0 \times \I_X^0 & \longrightarrow {\mathbb{G}}_m\\
(\alpha, \beta) & \mapsto \langle \alpha, \beta\rangle \,:=\,  \tilde{\alpha}\tilde{\beta}\tilde{\alpha}^{-1}\tilde{\beta}^{-1}
\end{aligned}		
\end{equation}
where $\tilde{\alpha}, \tilde{\beta}\in \tilde{\I}^0_X$ are preimages of $\alpha$, $\beta$ respectively.

\begin{thm}[Weil Reciprocity Law]\label{t:WRL}
	The central extension 
		\begin{equation}\label{eq:WRL}
		0\to {\mathbb{G}_m} \to \tilde{\Sigma^*_X} \to \Sigma^*_X \to 0
		\end{equation}
	constructed from \eqref{eq:centralext-global} through $ \Sigma^*_X \hookrightarrow \I^0_X$ is trivial.
\end{thm}

\begin{proof}
The key observation is that $\Sigma_X$ is a point in $\gr(\A_X)$ which is fixed under the action of $\Sigma_X^*$ (see \cite[Thm 3.5]{MP} for details).  
\end{proof}

\begin{cor}\label{c:WRL}
Under these hypothesis, the commutator of two rational functions is trivial; that is,
		\[
		1  \,=\, \langle f, g \rangle \qquad \forall f,g \in \Sigma_X^*
		\] 
or, what is tantamount, 
	\begin{equation}\label{e:CC}
	1 \,=\, \prod_{x\in X} (-1)^{v_x(f) v_x(g)\deg(x)} 
	 \operatorname{Norm}_{\k(x)/\k} \Big( \frac{f^{v_x(g)}}{g^{v_x(f)}}(x)\Big)
	\end{equation}
	for all $f,g\in \Sigma_X^*$.
\end{cor}

\begin{proof}
Theorem~\ref{t:WRL} shows that the central extension~\eqref{eq:WRL} is trivial. Thus, the associated cocycle is (cohomologus to) the trivial cocycle and, in particular, the commutator is trivial. Bearing in mind  Proposition~\ref{p:Com-cocycle} and Theorem~\ref{t:wx=1} and the relation $w_x(f)=v_x(f)\deg(x)$, one concludes. 
\end{proof}

%

\begin{rem}
The previous proofs rely on the study of central extensions (see \cite{Brylinski} for related ideas). However, there are other approaches that yield similar results. For instance, the expression \eqref{e:CC} coincides with the given in \cite{ContouCarrere} when dealing with a formal approach to the tame symbol. In \cite{AP} the authors use an alambicated computation of infinite determinants.
\end{rem}

\subsection{Theorem of Residues}\label{ss:TR}
We will show how the Theorem of Residues can be deduced as a direct consequence of the Weil Reciprocity Law. The key idea is that we have shown in \S\ref{ss:Lie} that all constructions so far (central extensions, cocycles, explicit expressions, etc.) hold true when groups are replaced by functors on groups.   

We  now consider the Lie algebras associated to the groups of \S\ref{ss:WRL} or, what is tantamount, we  regard the constructions of \S\ref{ss:WRL} as an statement for functors on groups and we take values in $R=\k[\epsilon]/(\epsilon^2)$. Note that $\Lie \I_X^0 \simeq \A_X$ and $\Lie{\mathbb{G}}_m = \k$, and that extension \eqref{eq:Ext-Lie} for the case of $G= \I_X^0$  reads as follows
	\begin{equation}\label{eq:extLieAdele}
	0\to \k \to \tilde{\A}_X \to \A_X \to 0
	\end{equation}
and the pairing \eqref{e:pairing-I} yields 
\begin{equation}\label{e:pairing-A}
\begin{aligned}
\A_X \times \A_X   & \longrightarrow \k \\
(\alpha, \beta) & \mapsto \c_{Lie}( \alpha, \beta)
\end{aligned}		
\end{equation}
where $1+\epsilon_1\epsilon_2 \c_{Lie}( \alpha, \beta) = \langle 1+ \epsilon_1\alpha, 1+\epsilon_2 \beta\rangle$. 

Note that  $\c_{Lie}$ defines a skew-symmetric bilinear pairing in $\A_X$ which will be called \emph{residue pairing}. The following result shows how the global object (for $X$) $\c_{Lie}$ can be expressed in terms of local ones (for each $x\in X$). 

\begin{thm}\label{t:Res-adele}
Let $\alpha,\beta\in \A_X$. It holds that
	\[
	\c_{Lie}( \alpha, \beta)
	\,=\, \sum_{x\in X}\tr_{\k(x)/\k}  \res_x \alpha \d\beta
	\]
where $\tr$ is the trace and $\res_x:K_x\simeq \k(x)(\!(z_x)\!)  \to \k(x)$ maps $\alpha_x=\sum_i a_i z_x^i\in\k(x)(\!(z_x)\!) $ to $a_{-1}$ ($\res_x$ does not depend on the choice of $z_x$, a formal parameter  at $x$). 
\end{thm}

\begin{proof}
First of all, note that \S\ref{subsec:extensions}  implies that \eqref{eq:Idele-ext} and \eqref{e:pairing-I}  do hold when groups are replaced by their Lie algebras. Accordingly, the cocycle of Lie algebras associated to \eqref{e:pairing-A},  $\c_{Lie}$,  and  the cocycle associated to \eqref{e:pairing-I}, $\c$, and the commutator $\langle\, ,\, \rangle$, are related by \eqref{eq:Lie-cocycle}
		\[ 
		1  + \epsilon_1\epsilon_2 \c_{Lie}( \alpha, \beta)
		\,=\,   \langle  1+\epsilon_1 \alpha, 1+ \epsilon_2 \beta   \rangle  \, . 
		\]

The result will follow from the explicit expansion of r.h.s. with the help of equation \eqref{eq:wind}, Proposition~\ref{p:Con-Carr} and equation~\eqref{eq:dirsum-cocycle}. 

Let $x\in X$ be a point and $z_x$ a formal parameter at $x$. Let $\alpha_x= \sum_{i} a_i z_x^i$ be the germ of $\alpha$ at $x$. Analogously, let $\beta_x=\sum_{i} b_i z_x^i$. Recalling \eqref{eq:wind}, it follows
	\[
	w_x(1+\epsilon\alpha_x) \,=\, 
	\begin{cases}
	0 & \text{ if }v_x(\alpha_x)\geq 0 \\
	v_x(\alpha_x) & \text{ if }v_x(\alpha_x) < 0
	\end{cases}
	\]

On the other hand, note that $1+\epsilon_1 \alpha_x = \prod_{i>0} (1+\epsilon_1 a_{-i} z_x^{-i} )  \prod_{i\geq 0}(1+\epsilon_1 a_i z_x^i)$. Then, having in mind that $\epsilon_1^2=\epsilon_2^2=0$, Proposition~\ref{p:Con-Carr} implies that 
	\[
	\begin{aligned}
	\langle 1+\epsilon_1 \alpha_x, 1+\epsilon_2 \beta_x \rangle
	\,& =\, \operatorname{Norm}_{\k(x)/\k} \big( 1- \epsilon_1\epsilon_2 \sum_i i a_i b_{-i} \big)
	\,=\\ & = \, 1+  \epsilon_1\epsilon_2 \tr_{\k(x)/\k} \res_x \alpha \d \beta \, . 
	\end{aligned}\]
Putting everything together, the claim is proved. 
\end{proof}

\begin{cor}\label{c:Tate}
	Let $x$ be a rational point and let $V:=\k(\!(z_x)\!) \simeq K_x, V^+:=V=\k[\![z_x]\!] \simeq  A_x, V^-:=z_x^{-1} \k[z_x^{-1}]$. Let $f_1,f_2\in V\setminus \{0\}$ acting as homotheties on $V$. It holds that:	
	\[
	\res_x f_1 \d f_2 \, =\, \tr(\gamma_2\beta_1-\gamma_1\beta_2)
	\]
where $\gamma_i:= \pi_{V^+}\circ (f_i\vert_{V^-})$, $\beta_i:= \pi_{V^-}\circ (f_i\vert_{V^+})$ and $\pi_{V^+}, \pi_{V^-}$ denote the projections. 
\end{cor}

\begin{proof}
	Arguments similar to those of the proof of Theorem~\ref{t:Res-adele}
	 combined with Corollary~\ref{c:cLie=tr} prove the claim. 
\end{proof}

\begin{rem}
	An alternative proof of Theorem~\ref{t:Res-adele} can be carried out using Corollary~\ref{c:cLie=tr}. Corollary~\ref{c:Tate}  was used by Tate as the definition of the residue \cite{Ta}.
\end{rem}

\begin{thm}[Theorem of Residues]\label{t:thm-res}
Under these hypothesis, it holds that
	\[
	0 \,=\,  \sum_{x\in X} \tr_{\k(x)/\k}  \res_x f \d g
	\]
	for all $f,g\in \Sigma_X$.
\end{thm}

\begin{proof}
Restricting the central extension \eqref{eq:extLieAdele} by the diagonal embedding $\Sigma_X\hookrightarrow\A_X$ allows us to construct a central extension of $\Sigma_X$ (as Lie algebras)
		\begin{equation}\label{eq:WRL-Lie}
		0\to \k \to \tilde{\Sigma_X} \to \Sigma_X \to 0
		\end{equation}

By the very construction, this extension coincides with the central extension of Lie algebras associated to the central extension of groups \eqref{eq:WRL}. Being the latter trivial (by Theorem~\ref{t:WRL}), i.e. $\c=1$, it follows that \eqref{eq:WRL-Lie} is also trivial; i.e. $\c_{Lie}=0$. Pluging this into Theorem~\ref{t:Res-adele}, the result follows. 
\end{proof}
%
%
%
%
%

Having in mind Theorem~\ref{t:Res-adele} and being $\k$ perfect,  one easily checks  that the radical of the residue pairing is $ \prod_{x\in X} \k(x)\subset \A_X$. Consider  the vector space defined by the orthogonal of $\Sigma_X$ with respect to the residue pairing
	\begin{equation}\label{e:SigmaPerp}
	\Sigma_X^{\perp}\,:=\, 
	\Big\{ \alpha\in \A_X \,\text{ s.t. } 
	\sum_{x\in X}  \tr_{\k(x)/\k} \res_x \alpha \d g = 0\quad\forall g\in\Sigma_X
	\Big\}\, .
	\end{equation}

The following result can be thought of as an algebraicity criterion for adeles based in the Theorem of Residues. 

\begin{thm}\label{t:algebraicity}
	Let $\k$ be algebraically closed. It holds that 
	\[
	\Sigma_X^{\perp}  \, =\, \Sigma_X \,+\,  \prod_{x\in X} \k(x)
	\, .
	\]
\end{thm}

\begin{proof}
	The Theorem~\ref{t:thm-res} shows that the r.h.s. in contained into the l.h.s.~. Now, let $\alpha $ belong to the l.h.s.~. Then, the linear map 
		\[
		\begin{aligned}
		\bar{\alpha}: \A_X &  \longrightarrow \, \k\\
		\beta &\longmapsto  \c_{Lie}( \alpha, \beta)  =\sum_{x\in X}  \tr_{\k(x)/\k} \res_x \alpha \d\beta
		\end{aligned}
		\]
	factorizes as a map $\A_X/\Sigma_X\to \k$. It is also straightforward that it factorizes by a map
		\[
		\A_X/(\Sigma_X+U_D) \,\longrightarrow\,  \k
		\]
	where $D$ is an effective divisor on $X$ such that $\sum_{x\in X} v_x(\alpha) x + D\geq 0$ and 
		\[
		U_D:=\{ \beta \,\text{ s.t. } \sum_{x\in X} v_x(\beta) x \geq D \}\, .
		\]
		
	Let us give a geometrical interpretation of $(\A_X/(\Sigma_X+U_D))^*$. Consider the exact sequence of $\O_X$-modules
		\[
		0\to \O_X \to \Sigma_X \to \Sigma_X/\O_X \to 0
		\]
	and tensor it with $\O_X(-D)$. The  associated long exact sequence of cohomology reads
		\[
		{\small \xymatrix@C=12pt{
		0 \ar[r] &  H^0(X,\O_X(-D)) \ar[r] \ar@{=}[d] &  \Sigma_X  \ar[r] \ar@{=}[d] &  
		H^0(X,\Sigma_X/\O_X(-D))  \ar[r] \ar@{=}[d]  &  
		H^1(X,\O_X(-D))\ar[r] \ar@{=}[d]  &  0
		\\
		0 \ar[r] & U_D \cap \Sigma_X \ar[r] &  \Sigma_X \ar[r] &   \A_X/U_D  \ar[r] &   \A_X/(U_D+\Sigma_X) \ar[r] &0  }}\]
Due to the hypotheses on $\k$ and on $X$, it holds that the dualizing sheaf is given by $\omega_X$, the canonical sheaf.  
Thus, the linear map $\bar{\alpha}$ is given by a meromorphic differential $\omega\in H^0(X,\omega_X(D)) \simeq  H^1(X,\O_X(-D))$ by the relation $\bar{\alpha}(\beta)= - \sum_{x\in X}\res_x \beta \omega$. That is, $\d\alpha=\omega$. 
	
	Now, let us consider a non-empty open subset $U\subset X$ and a rational function $f\in \Sigma_X$ such that $\d f =\omega$. Since $f$ belongs to the l.h.s., it follows that $\alpha-f $ also belongs to the l.h.s. or, what is tantamount, we may assume that there is a non-empty open subset $U$ such that $d(\alpha-f)\vert_U=0$. Since $\prod_{x\in X} \k(x)$ lies also in the l.h.s. we may assume without lost of generality that  $\alpha\vert_U=0$. Accordingly, for such $\alpha$ one has
		\[
		\sum_{x\in X\setminus U} \tr_{\k(x)/\k}  \res_x \alpha \d g \,=\, 0 \qquad \forall g\in \Sigma_X
		\]
But this condition implies that $\alpha_x\in\k(x)$ for all $x\in X\setminus U$ and the conclusion follows. 
\end{proof}

\begin{rem}
A multiplicative analog of definition \ref{e:SigmaPerp} is introduced in terms of the pairing \ref{e:pairing-I} as follows
	\[
	\Sigma_X':= \Big\{ \alpha\in \I_X^0 \,\text{ s.t. } 
	\langle \alpha , f \rangle =1 \quad\forall f\in\Sigma_X^*
	\Big\}\, .
	\]	
By Corollary \ref{c:WRL}, it makes sense to consider the quotient $\Sigma_X'/\Sigma_X^*$. Hence, one wonders whether Theorem \ref{t:algebraicity} has a multiplicative version from which it could be deduced. Indeed, such a study is the goal of \cite{MNPP}.
\end{rem}
%
%


\subsection{Gelfand-Fuchs cocycle}\label{ss:remarks}

In the last decades have been a renewed interest in these results due to their connections with mathematical physics and the geometric Langlands Program. Let us say a couple of words about this issue. 

Let us fix be a natural number  $n$ and a subgroup $G\subseteq \Gl(n,\k)$. Let $\g:=\Lie G$. Since $\g$ can be identified with $G(\operatorname{Spec}\k[\epsilon]/\epsilon^2)\times_{G(\operatorname{Spec} \k)}\{1\}$, which is commutative w.r.t. the group law inherited from $G$, one can apply most results of the paper. 

On the one hand, we may consider the \textit{local case}. That is, we apply the results of \S\ref{subsec:extensions}
 to the case $(V=K_x^{\oplus{n}}, V^+ = (A_x^+)^{\oplus{n}})$. Let us consider the \emph{loop group} 
	\[
	L_x\g \,:=\, \g \otimes_{\k} K_x \,\subseteq \, \gl(n, K_x)\, . 
	\]
The extension of this Lie algebra induced by \eqref{eq:TildeGl} is precisely the Kac-Moody algebra. Further, arguing as in the proof of Theorem~\ref{t:Res-adele}, one has
	\begin{equation}\label{e:GF}
	\begin{aligned}
	L_x\g \times 	L_x\g & \longrightarrow \k \\
	(S\otimes \alpha_x, T \otimes \beta_x) & \longmapsto 
	 \tr_{\k(x)/\k} \res_x  \tr(S T) \alpha_x \d \beta_x
	\end{aligned}
	\end{equation}
where $\tr$ denotes the trace map of $ \gl(n, K_x)$.  It is worth noticing that it coincides with the well-known Gelfand-Fuchs cocycle. See \cite[\S1.3.5, \S3.15]{Frenkel} for the relation with loop algebras. 

On the other hand,  the study of the \textit{global case} will yield a generalization of the Theorem of Residues. Let $V=\g\otimes_{\k}  \A_X $ and $V^+=\g\otimes_{\k} \A_X^+$.  The commutator  endows $\g\otimes_{\k} \A_X$ with a skew-symmetric bilinear map as in \eqref{e:pairing-A} which is expressed in terms of \eqref{e:GF}
\[
\begin{aligned}
\g\otimes_{\k}  \A_X  \times \g\otimes_{\k}  \A_X   & \longrightarrow \k \\
(S\otimes \alpha, T\otimes \beta) & \longmapsto 
	\sum_{x\in X}  \tr_{\k(x)/\k} \res_x \tr(S T) \alpha \d\beta
	\, . \end{aligned}
\]

	
%
%

\begin{thm}
Define $L_X\g:=   \g \otimes_{\k} \Sigma_X  \subset   \gl(n, \Sigma_X)$. It holds that
	\[
	\sum_{x\in X}  \tr_{\k(x)/\k} \res_x  \tr(S T) f \d g \, =\, 0
	\]
for $ S\otimes f, T \otimes g \in L_X\g$.
\end{thm}

\begin{proof}
	One checks that $\g \otimes_{\k} \Sigma_X$ is a point of the infinite grassmannian of $\g\otimes_{\k}  \A_X $. Since this point is invariant under the action of $G \times \Sigma_X^*$, where $G$ acts on $\g$ via the adjoint action, it follows that the central extension defined by the determinant becomes trivial when restricted to $G \otimes_{\k} \Sigma_X^*$ (see \S\ref{subsec:extensions}). Therefore, the corresponding extension at the level of Lie algebras, as in \S\ref{ss:Lie}, is trivial too. 
	
	Expressing the corresponding cocycle in terms of the cocycles of the local cases (as it was done in Theorem~\ref{t:Res-adele}), the conclusion follows. 
\end{proof}

%



\end{document}